\documentclass[a4paper,10pt,twoside]{amsart}
\usepackage{amsmath,amssymb,amsthm,hyperref,mathtools,esint,enumitem,accents}
\usepackage[svgnames]{xcolor}
\hypersetup{linktoc=all,colorlinks=true,linkcolor=Brown,citecolor=Maroon,urlcolor=DarkBlue} 

\usepackage[margin=1.5in, marginparwidth=1.2in]{geometry}

\newcommand{\dx}[1]{\,\mathrm{d}#1}  

\DeclareMathAlphabet{\mathlcal}{U}{dutchcal}{m}{n} 

\usepackage{tikz-cd}
\usepackage{cases}
\makeatletter
\@namedef{subjclassname@2020}{\textup{2020} Mathematics Subject Classification}
\makeatother
\usepackage[capitalize]{cleveref}
\newtheorem{theorem}{Theorem}[section]
\newtheorem*{theorem*}{Theorem}

\newtheorem{lemma}[theorem]{Lemma}
\theoremstyle{remark}
\newtheorem{remark}[theorem]{Remark}
\theoremstyle{definition}

\newtheorem{assumption}[theorem]{Assumption}

\newtheorem*{duplicate*}{{\hypersetup{hidelinks}Assumption~\ref{ass}}}

\newtheorem*{nota*}{Notation}

\Crefname{assumption}{Assumption}{Assumptions}

\usepackage{autonum} 

\numberwithin{equation}{section}

\usepackage{etoolbox} 
\expandafter\patchcmd\csname\string\proof\endcsname
  {\normalparindent}{0pt }{}{}

\newcommand{\T}{\mathbb{T}}

\newcommand{\eps}{\varepsilon}

\newcommand{\id}{{\rm id}}

\newcommand{\N}{\mathbb{N}}

\newcommand{\R}{\mathbb{R}}

\newcommand{\Z}{\mathbb{Z}}

\newcommand{\cP}{\mathcal{P}}

\DeclareMathAlphabet{\mathup}{OT1}{\familydefault}{m}{n}

\title[Sharp mean-field estimates for the repulsive log gas]{Sharp mean-field estimates for the repulsive log gas in any dimension}
\author[M. G. Delgadino]{Matias G. Delgadino}
\address{Department of Mathematics, University of Texas at Austin, Austin, TX 78712, USA}
\email{matias.delgadino@math.utexas.edu}
\author[R. S. Gvalani]{Rishabh S. Gvalani}
\address{Department of Mathematics, ETH Zürich, Rämistrasse 101, 8092 Zürich, Switzerland}
\email{rishabhsunil.gvalani@math.ethz.ch}
\date{\textcolor{Maroon}{\today}}
\subjclass[2020]{} 
\allowdisplaybreaks
\begin{document}
\begin{abstract}
We prove sharp estimates for the mean-field limit of weakly interacting diffusions with repulsive logarithmic interaction in arbitrary dimension. More precisely, we show that the associated partition function is uniformly bounded in the number of particles $N$ for an arbitrary bounded base measure. Combined with the modulated free energy method, this amounts to a logarithmic improvement in $N$ of the current best available closeness estimates in the literature. Our arguments are inspired by and borrow ideas from Nelson's classical construction of the $\varphi^4_2$ Euclidean quantum field theory.
\end{abstract}
\maketitle

\tableofcontents
\section{Introduction} 
This article aims to obtain sharp estimates for the mean-field convergence of interacting particle systems with repulsive logarithmic interaction in arbitrary dimension, improving the best available estimates \cite{BJW1,serfaty2020mean,cai2024propagation} by a logarithm in the number of particles. More specifically, we consider a large number $N \geq 1$ of identical interacting particles $X_t=(X_t^1,\dots,X_t^N) \in \Omega^N$ in some domain $\Omega=\T^d,\R^d$, whose motion is described by the following system of interacting stochastic differential equations (SDEs)
\begin{equation}\label{particles}
    dX_t^i=-\nabla V(X_t^i)-\frac{1}{N}\sum_{j\ne i,j=1}^N \nabla_1 W(X_t^i,X_t^j)+\sqrt{2}\, dW_t^i\,
\end{equation}
where $(W_t^i)_{i=1}^N$ is a family independent Brownian motions,  $V:\Omega\to\R$ is a confining potential, and $W:\Omega\times\Omega\to \R$ is a pairwise interaction potential which is symmetric $W(x,y)=W(y,x)$.

The convergence of exchangeable particle systems to their mean-field limit is a classical topic that has attracted the attention of mathematicians for decades. We refer the reader to \cite{Sznitman1991} for a classical reference and \cite{ChaintronDiez2022a,ChaintronDiez2022b} for a more modern review. It is well-known that under regularity conditions on the potentials, the law of $X_t\sim \rho^N_{t}\in C([0,\infty);\mathcal{P}(\Omega^N))$ can be well approximated by $N$ independent copies the McKean--Vlasov equation. More specifically, we have the bound
\begin{equation}\label{eq:meanfield}
  \overline{\mathcal{H}}(\rho^N_t|\bar{\rho}^{\otimes N}_t)\le e^{Ct}\left(\overline{\mathcal{H}}(\rho^N_{0}|\bar\rho^{\otimes N}_{0})+\frac{1}{N}\right) \, ,
\end{equation}
where $\overline{\mathcal{H}}(\cdot|\cdot)$ is the dimension-averaged relative entropy and $\rho^{}_t$ is the solution of McKean--Vlasov equation
\begin{equation}
\label{eq:MVpde}
\begin{cases}
\partial_t \bar\rho_t = \Delta \bar\rho_t + \nabla \cdot (\bar\rho_t(\nabla W \star \bar\rho_t+\nabla V)) \, ,\\
\bar\rho_0=\bar\rho_{0}.
\end{cases} 
\end{equation}
Note that we use `$\star$' to denote the following generalised version of the convolution:
\begin{equation}
  (W \star \eta)(x)\coloneqq \int_{\Omega} W(x,y) \, \mathrm{d}\eta(y) \, , 
\end{equation}
for any distribution $\eta$, whenever the above pairing is well-defined. Since we wish to focus only on the singularity of $W$, we will assume throughout the text that $W$ is repulsive in the sense of $H$-stability: for any $\varphi \in C^\infty_c(\Omega)$ with $\int_{\Omega} \varphi =0$, we have that
\begin{equation}
 \int_{\Omega^2} W(x,y) \varphi(x)\varphi(y) \, \dx x\,\dx{y} \geq 0 \, . 
\end{equation} 
In a companion paper, we will show how one can extend some of the results of this article to the case of attractive interactions (see \cite{Delgadino2025}).

In recent years, several advances have been made in the derivation of mean-field limits for singular interactions. We mention in particular the seminal work of Jabin and Wang \cite{jabin2018quantitative} who, using relative entropy methods, were able to consider interactions given by divergence-free vector fields in $W^{-1,\infty}$. A few years later, a breakthrough of Serfaty~\cite{serfaty2020mean} showed a quantitative mean field limit for the repulsive log gas, even in the case without noise. Harnessing this breakthrough Bretsch, Jabin, and Wang \cite{BJW1} were able to combine the two above methods into  the so-called modulated free energy method.

To illustrate the novelty of our results, we now focus on the specific case of the two-dimensional repulsive log gas case (see the original references \cite{Dys62,Meh04,Mar88,MY80} and the lecture notes \cite{serfaty2024lectures} for an introduction). More specifically,  we consider $\Omega=\T^2$, a vanishing confining potential $V=0$, and the interaction potential  given by $W(x,y)=F(x-y)=(-\Delta)^{-1}(\delta_0-1)(x-y)$, the Green's function of the Laplacian on $\T^2$ with zero average. In this case, it was shown by \cite{serfaty2020mean,BJW19a} that 
\begin{equation}\label{eq:meanfieldlog}
\overline{\mathcal{H}}(\rho^N_t|\bar\rho^{\otimes N}_t)\le e^{Ct}\left(\overline{\mathcal{H}}(\rho^N_{0}|\bar\rho^{\otimes N}_{0})+\frac{\log N}{N}\right)\, .
\end{equation}
We note that the difference between \eqref{eq:meanfield} and \eqref{eq:meanfieldlog} is given by the factor $\log N$ in the right-hand side, which stems from the singularity of the interaction potential. In this article we improve \eqref{eq:meanfieldlog} by getting rid of the logarithmic factor, thus recovering the rate in the smooth interaction case \eqref{eq:meanfield}. We follow the modulated free energy approach of the cited references \cite{serfaty2020mean,BJW19a}, with the improvement stemming from a sharper lower bound on the modulated free energy. Even though the exposition here focuses on $\T^2$ and the Coulomb gas, our proof works for any setting in which the regularity allows for the modulated free energy method to be applied as our bounds rely only on the $L^\infty$-norm on $\bar{\rho}_t$.

 \subsection{Modulated free energy method}
We give a short outline of the  modulated free energy method. We refer the reader to the lecture notes Bretsch, Jabin, and Wang \cite{BJW19a} for more details. The starting point is to consider the modulated free energy
\begin{equation}\label{eq:modulatedenergy}
F^N(t)=\overline{\mathcal{H}}(\rho^N_t|\bar\rho^{\otimes N}_t)+\frac{1}{N}\int_{\Omega^N}\mathring{I}_W[\eta^N_x]\;\dx \rho^N_t(x),
\end{equation}
where, for any $x=(x_1,\dots,x_N)\in \Omega_N$, we define
\begin{equation}
  \eta^N_x \coloneqq N^{-\frac12}\left(\sum_{i=1}^N \delta_{x_i}-\bar\rho\right) \, , \quad \mathring{I}_W[\eta]=\frac{1}{2}\int_{\Omega^2\setminus \Delta} W(x,y)\;\dx\eta(x)\dx\eta(y)\, ,
  \label{eq:fluctuation}
\end{equation}
with $\Delta=\{(x,y)\in \Omega:x=y\}$.
Differentiating the free energy, one obtains the estimate
\begin{equation}\label{eq:dissipation}
  \frac{\rm d}{\rm dt}F^N(t)\le -\frac{1}{N} \int_{\Omega^N}\mathring{I}_{Q_{\bar\rho_t}}[\eta^N_x]\;\dx \rho^N_t(x),
\end{equation}
where 
$$
Q_{\bar\rho_t}(x,y)=\nabla_1 W(x,y)\cdot \left(\nabla\log\bar\rho_t(x)-\nabla\log\bar\rho_t(y)+\nabla W\star \bar{\rho}_t(x)-\nabla W\star \bar{\rho}_t(y)\right)\, ,
$$
(see \cite[Equation 2.6]{BJW19a}). The original observation of \cite{Jabin2016} was to notice that under a Lipschitz bound on $\log\bar\rho_t$ (see for instance \cite[Theorem 3.1]{cai2024propagation}), and the bound $|\nabla W(x,y)|\le \frac{1}{|x-y|}$, which is the case of the logarithmic potential, we have that 
\begin{equation}\label{eq:Qbound}
\|Q_{\bar\rho_t}(x,y)\|_\infty<\infty.
\end{equation}
Following \cite{Jabin2016}, we use the Donsker--Varadhan bound for the right hand of \eqref{eq:dissipation}
\begin{equation}\label{eq:inequality}
  -\frac{1}{N} \int_{\Omega^N}\mathring{I}_{Q_{\bar\rho_t}}[\eta^N_x]\;\dx \rho^N_t(x)\le \frac{1}{\delta}\left(\overline{\mathcal{H}}(\rho^N_t|\bar\rho^{\otimes N}_t) +\frac{1}{N}\underbrace{\log \left(\int_{\Omega^N} e^{-\delta \mathring{I}_{Q_{\bar\rho_t}}[\eta^N_x]} \;\dx\bar\rho^{\otimes N}(x)\right)}_{R_\delta(t)}\right).
\end{equation}
Using \eqref{eq:Qbound}, we can apply \cite[Theorem 4]{Jabin2016} to say that there exists $\delta>0$ such that
\begin{equation}\label{eq:boundR}
  R_\delta(t)\le C.
\end{equation}
Applying \eqref{eq:boundR} and \eqref{eq:inequality} to \eqref{eq:dissipation}, we obtain the estimate
\begin{equation}\label{eq:dissipation2}
  \frac{\dx}{\dx{t}}F^N(t)\le C\overline{\mathcal{H}}(\rho^N_t|\bar\rho^{\otimes N}_t) +\frac{C}{N}.
\end{equation}
To complete a standard Gronwall argument, we need to lower bound the modulated free energy $F^N(t)$ by the relative entropy $\overline{\mathcal{H}}(\rho^N_t|\bar\rho^{\otimes N}_t)$. So far, a control of this type stems from pointwise bound shown by Serfaty in \cite[Corollary 3.5]{serfaty2020mean}
\begin{equation}\label{eq:pointwise}
  -\frac{1}{2N}\mathring{I}_W[\eta^N_x]\le C \frac{\log N}{N},
\end{equation}
where $C$ depends on the upper bound of the density of $\bar{\rho}_t$. The bound \eqref{eq:pointwise} is responsible for the logarithm in the final estimate \eqref{eq:meanfieldlog}. A contribution of this paper is to show that the bound \eqref{eq:pointwise} is actually not saturated in an integral sense, and we can obtain the sharp estimate \eqref{eq:meanfield} for repulsive logarithmic interactions. We now state the main result of this paper which provides an improvement of this bound for general base measures $\bar{\rho}\in \cP(\Omega)$ with bounded densities (see~\cref{thm:main} for a precise version).
\begin{theorem}\label{lem:energybound}
 Given $\bar\rho\in\mathcal{P}(\Omega)$ and $\beta>0$, there exists 
 \[
   C=C\left(\beta,\left \lVert \frac{\dx{\bar{\rho}}}{\dx{x}} \right \rVert_{L^\infty(\Omega)}\right) \, ,
 \] 
 such that,  for every $N\in\N$ and $\rho^N\in\mathcal{P}(\Omega^N)$, we have the bound
\begin{equation}\label{eq:integral}
  -\frac{\beta}{N}\int_{\Omega^N}\mathring{I}_W[\eta^N_x]\;\dx \rho^N(X)\le \frac{1}{2}\overline{\mathcal{H}}(\rho^N|\bar\rho^{\otimes N})+\frac{C}{N}.
\end{equation}
\end{theorem}

Integrating \eqref{eq:dissipation2} and using \eqref{eq:integral}, we obtain the estimate
\begin{equation}\label{eq:finalestimate}
  \overline{\mathcal{H}}(\rho^N_t|\bar\rho_t^{\otimes N})\le C\int_0^t \overline{\mathcal{H}}(\rho^N_s|\bar\rho_s^{\otimes N})\;\dx{s}+\overline{\mathcal{H}}(\rho^N_0|\bar\rho_0^{\otimes N}) +\frac{C}{N} \, .
\end{equation}
By a standard Gronwall argument, we obtain the same closeness estimate \eqref{eq:meanfield} as in the smooth case. This is sharp in its dependence on $N$, improving directly the known results on for the repulsive log case by a factor of $\log N$. \cref{lem:energybound} can also be utilized to obtain a logarithmic improvement in $N$ of the estimates of uniform-in-time propagation or generation of chaos. We refer the interested reader to the results in \cite{cai2024propagation,rosenzweig2023modulated}.

\subsection{Method of proof} 
As in \eqref{eq:inequality}, we start by applying the Donsker--Varadhan bound to obtain
\begin{equation}
  -\frac{\beta}{N}\int_{\Omega^N}\mathring{I}_W[\eta^N_x]\;\dx \rho^N(x)\le \frac{1}{2}\overline{\mathcal{H}}(\rho^N|\bar\rho^{\otimes N}) -\frac{1}{2N}\log\left(\underbrace{\int_{\Omega^N}e^{-2\beta \mathring{I}_W[\eta^N_x]}\;\dx\bar\rho^{\otimes N}(x)}_{Z_{N,2\beta}}\right).
\end{equation}
Hence, the result follows if we show that for every $\beta>0$ the partition function 
\begin{equation}\label{eq:partbound}
\sup_{N\in\N}\int_{\Omega^N}e^{-\beta \mathring{I}_W[\eta^N_x]}\;\dx\bar\rho^{\otimes N}(x)=:\sup_{N\in\N}Z_{N,\beta}<\infty
\end{equation}
is uniformly bounded in $N$. The main difference with the case of \eqref{eq:inequality}, is that the interaction potential $Q_{\bar\rho_t}$ is bounded in $L^\infty$, while we allow the interaction potential $W$ to have a logarithmic singularity at the origin.

Classically, the proof of boundedness of partition functions of the type \eqref{eq:partbound} was restricted to the class of continuous interaction potentials (see Ben Arous and Brunaud \cite{BenArous1990}). Later Jabin and Wang \cite{Jabin2016} extended this to the case of potentials with small $L^\infty$ norm. To our knowledge, the only known result showing boundedness of the partition function \eqref{eq:partbound} that considers a potential $W$ with a logarithmic singularity is the work of Grotto and Romito \cite{grotto2020central}. This is done in the case where the interaction potential is exactly the Green's function of the Laplacian on $\T^2$ and the base measure $\bar\rho=\dx x$ is the Lebesgue measure.  This setting allows them to take advantage of the Fourier series representation of the potential and the fact that the base measure is translation invariant. They rely on splitting the Green's function into a short-range Yukawa potential and a long-range regular remainder and controlling their contributions individually. 

In this article, we are agnostic to the specific Fourier representation of the interaction potential and the choice of base measure. Instead, we exploit Besov regularity of the type $W\in \dot{B}^{\frac dp}_{p,\infty}$ for any $p\in(d,\infty)$ to obtain the bound on the partition function (see \cref{ass:potential} and \cref{sec:estimates_for_the_repulsive_log_gas}). A similarity between our approach and the one of Grotto and Romito \cite{grotto2020central} is that we relate this problem to techniques from singular SPDEs and constructive quantum field theory.  More specifically, our method of proof is based on the classical argument of Nelson \cite{nelson1966quartic} for the boundedness of the partition function of the $\varphi^4_2$ Euclidean quantum field theory. For a gentle introduction to Malliavan calculus and Nelson's argument, we recommend Hairer's lecture notes \cite{hairer2021introduction}. The proof presented here sidesteps Malliavan calculus and uses directly the mentioned Besov regularity of the interaction potential along with some explicit correlation inequalities.

\subsection{Nelson's construction of \texorpdfstring{$\varphi^4_2$}{phi4\_2}} 
\label{sub:method_of_proof}
In this section, we show how the structure of $Z_{N,\beta}$ is very similar to that of the partition function of the Euclidean $\varphi^4_2$ measure, i.e.
\begin{equation}
  Z= \mathbb{E}\left[e^{-\int_{\T^2}\gamma^4 \, \dx{x}}\right]  
\end{equation}
where $\gamma$ is sampled from the two-dimensional Gaussian free field. A priori the above quantity does not make sense. The Gaussian free field in $d=2$ is only supported on the Besov space $B^{-\eps}_{\infty,\infty}(\T^d)$ for every $\eps>0$ and, in particular, is not function-valued. Thus, the quantity $\int_{\T^2}\gamma^4\, \dx{x}=\infty$  is almost surely $+\infty$. Hence, constructing the $\varphi^4_2$-measure by bounding the partition function requires renormalisation. This amounts to regularising the free field by mollifying it at some length scale $\eps>0$ and then subtracting appropriately divergent terms. This is akin to removing the self-interactions of the interacting field. Effectively, one replaces $\gamma^4$ with the fourth Wick power $:\gamma^4:$. Using Malliavan calculus one can show that
\begin{equation}
  \sup_{\eps>0}Z_\eps \coloneqq \sup_{\eps>0}\mathbb{E}\left[e^{- \int :\gamma_\eps^4:\, \dx{x}}\right] <\infty \, ,
\end{equation}
thus allowing one to construct the limiting measure.

We now turn to the quantity of interest $Z_{N, \beta}$ and re-express it as an expectation as follows
\begin{equation}
  Z_{N,\beta}= \mathbb{E} \left[e^{-\beta\mathring{I}_W[\eta^N_X]}\right]\, ,
\end{equation}
where $X=(X_1,\dots,X_N)$ is a vector of $N$ independent $\Omega$-valued $\bar\rho$-distributed random variables, and $\eta^N_X$ is as defined in~\eqref{eq:fluctuation}. Thus, as $N\to \infty$,  $\eta^N_X$ converges in law to white noise or, more precisely, to a random variable $\eta_\infty$ whose law is the Gaussian measure with Cameron--Martin space $L^2(\Omega;\bar\rho)$. 

In the case of the Green's function of $\T^2$, by reintroducing the self-interactions in $\mathring{I}_W$, we obtain the $H^{-\frac{d}{2}}$-norm of $\eta^N_X$. As in the $\varphi^4_2$ model, $\eta^N_X$ does not belong to $H^{-\frac{d}{2}}$ for any $N < \infty$. Even in the limit $\eta_\infty$ is at best in $H^{-\frac{d}{2}-\eps}$ for every $\eps>0$, and so $\|\eta_\infty\|_{H^{-\frac{d}{2}}}=\infty$ is almost surely $+\infty$. However, the quantity of interest is $\mathring{I}_W$ in which we have ``removed''  the logarithmic singularity at the diagonal. This is analogous to the renormalisation with the fourth Wick power in $\varphi^4_2$ setting which is key to showing that the partition function is bounded.

In fact, our method of proof will resemble Nelson's proof for controlling $Z_\eps$. The point where our techniques diverge is in the proof of certain higher-order moment bounds. Nelson's argument relies on the fact that the reference measure is Gaussian, and exploits hypercontractivity to estimate higher moments by the second moment. Since the object $\eta^N_X$ is only asymptotically Gaussian we need to estimate correlations by hand. Our main technical contribution is a refined correlation inequality for the higher-order moments of $\mathring{I}_W$ (see \cref{lem:corineq}).

\section{Assumptions and main result} 
\label{sec:preliminaries_and_main_results}
In this section, we introduce the notation and assumptions for our main result \cref{thm:main}. Since we will focus only on bounding the object $Z_{N,\beta}$ for an appropriate choice of base measure $\bar{\rho},W$, we can be more liberal with our choice of domain $\Omega$. Our domain $\Omega$ will either be $\T^d$, $\R^d$, or a bounded open subset of $\R^d$ with smooth boundary. The results can then be applied to the dynamics whenever the modulated free energy method can be made to work,  which to our knowledge has only been studied for $\Omega=\T^d,\R^d$. However, working with a general $\Omega$ allows us to study the convergence of Gibbs measure on general domains (see~\cref{sub:gibbs_measure_and_convergence}).

We will use $P_t$, (resp. $P_t^{\frac 12}$) to denote the heat semigroup (resp. $1/2$-fractional heat semigroup) on $\Omega^2$ with Neumann boundary conditions (whenever $\Omega$ has a boundary). We will often abuse notation by using $P_t$ to denote both the semigroup and the associated kernel. We will also use $\dx{x}$ to denote the Lebesgue measure on both $\Omega,\Omega^N$. It will be clear from context which of the two we are referring to.  We now state the assumptions we make on the interaction potential $W$. 

\begin{assumption}
  \label{ass:potential0}
  We assume that we have an interaction potential $W \in L^1_{\rm loc}(\Omega^2)$ which is lower semicontinous and repulsive/H-stable, i.e. for all $\eps>0$
\begin{equation}
\int_{\Omega^2}(P_\eps W)(x,y)\dx{\eta}(x)\dx{\eta}(y)\geq 0 \, ,
\end{equation}
for all finite signed measures $\eta \in \mathcal{M}(\Omega)$ with $\eta(\Omega)=0$. 
\end{assumption}
\begin{assumption}\label{ass:potential}
We  make the following assumptions on the growth and regularity of $W$ for $d\geq 2$:
\begin{enumerate}
\item Logarithmic upper bound, i.e. there exists a constant $C_0<\infty$ such that for all $x \in \Omega$ and $\eps \in (0,1)$
\begin{equation}
|(P_\eps W)|(x,x)\leq C_0(|\ln|\eps||+1)\, .
\label{eq:logarithmicbound}
\end{equation}
\item Besov-type regularity: We assume that for every measure $\bar{\rho}$ with essentially bounded density and $p \in [1,\infty)$, there exists $\kappa>0$ and a constant $C<\infty$ such that
\begin{equation}
\sup_{\eps> 0}\sup_{x\in \Omega}\lVert (P_\eps- \id)W(x,\cdot) \rVert^p_{L^p(\Omega;\bar{\rho})} < C \eps^{\kappa} \, ,
\label{eq:Besovestimate}
\end{equation}
where
\begin{equation}
  C=C\left(p,\left \lVert\frac{\dx{\bar{\rho}}}{\dx{x}} \right \rVert_{L^\infty(\Omega)}\right)\, .
\end{equation}
\item Quantified super-harmonicity: We assume there exist constants $\alpha,\eps_0>0,K<\infty$ such that
\begin{equation}\label{superharmonicity}
W(x,y)-(P_\eps W)(x,y)\ge -K \eps^\alpha \qquad\mbox{for any $(x,y)\in\Omega^2$ and $t<t_0$} \, .
\end{equation}
\end{enumerate}
For $d=1$, we replace $P_\eps$ by $P_\eps^{\frac 12}$, to avoid facing issues which arise from taking the second derivative of the logarithm in one dimension.
\end{assumption}

\begin{remark}
  ~\cref{ass:potential} mimics the regularity and growth of the log gas, i.e. the choice of $W=(-\Delta)^{-d/2}\delta$, while~\cref{ass:potential0} ensures that the interaction is repulsive. We will show that this example satisfies all the above assumptions in \cref{sec:estimates_for_the_repulsive_log_gas}.
\end{remark}
\begin{remark}
    For $d\ge 2$, the condition~\eqref{superharmonicity} is satisfied as long as the negative part of the Laplacian has some  integrability
    $$
    \sup_y\|\min\{-\Delta_x W(x,y),0\}\|_{L^p(\Omega)}<\infty \qquad\mbox{for some $p>d/2$.}
    $$
    In this case, we have that $\alpha=2-d/p$. 
    
    For $d=1$, we have to consider the $1/2$-Laplacian instead, and check the analogous condition
    $$
    \sup_y\|\min\{(-\Delta_x)^{1/2} W(x,y),0\}\|_{L^p(\Omega)}<\infty \qquad\mbox{for some $p>d$,}
    $$
    to obtain $\alpha=1-d/p$.
\end{remark}

We are now in a position to state our main result.

\begin{theorem}\label{thm:main}
  Suppose \cref{ass:potential0,ass:potential} are satisfied. Then, for any $\beta>0$ and $\bar{\rho}\in \cP(\Omega)$ with essentially bounded density there exists some
\[
  C=C\left(\beta,\left \lVert\frac{\dx{\bar{\rho}}}{\dx{x}} \right \rVert_{L^\infty(\Omega)}\right) \, ,
\]
such that for every $N\in\N$
  \begin{equation}
    1\le Z_{N,\beta}=\int_{\Omega}e^{-\beta\mathring{I}_W[\eta^N_x]}\;\dx\bar\rho^{\otimes N}(x)\le C.
  \end{equation}
  where 
  \begin{equation}
  \eta^N_x \coloneqq N^{-\frac12}\left(\sum_{i=1}^N \delta_{x_i}-\bar{\rho}\right) \, , \quad \mathring{I}_W[\eta]=\frac{1}{2}\int_{\Omega^2\setminus \Delta} W(x,y)\;\dx\eta(x)\dx\eta(y)\, .
\end{equation}
As an immediate consequence, we have that~\eqref{eq:integral} holds true. 
\end{theorem}

\section{Proof of~\cref{thm:main}} 
\label{sec:proofs}
Without loss of generality, we set $\beta=1$. We use the shorthand notation $W_\eps:=P_\eps W$ for the regularized interaction potential, and define
\begin{equation}
  \eta^N_x= N^{-\frac12}\left(\sum_{i=1}^N \delta_{x_i}-\bar\rho\right)\, ,
\end{equation}
as in \eqref{eq:fluctuation}. We will use the notation $\mathring{I}_W[\eta^N_x]$ to denote the interaction energy defined in~\eqref{eq:fluctuation} with respect to the potential $W$, and $\mathring{I}_{W_\eps}[\eta^N_x]$ for the regularized potential $W_\eps$.

We start the proof by proving the folowing lemma that allows us to control the behaviour of $\mathring{I}_W[\eta^N_x]$. 
\begin{lemma}
Assume $W $ satisfies~\cref{ass:potential0,ass:potential} with $\bar\rho \in \cP(\Omega)$ as in the statement of~\cref{thm:main}. Then, 
\begin{enumerate}
\item[(a)] there exists a constant $C_1\in (0,\infty)$ such that $ \mathring{I}_W[\eta^N_x]\geq -C_1\ln N$,
\item[(b)] and there exists a constant $C_2\in (0,\infty)$ such that $\mathring{I}_{W_\eps}[\eta^N_x]\geq -C_2|\ln \eps|$ for $\eps \in (0,1/2)$.
\end{enumerate}
\label{lem:coulombapriori}
\end{lemma}
\begin{proof}
  We first prove (b).  Using \eqref{eq:logarithmicbound}, we obtain
$$
\mathring{I}_{W_\eps}[\eta^N_x]= I_{W_\eps}[\eta^N_x]+\frac{1}{N}\sum_{i=1}^N W_\eps(x_i,x_i).
$$
We now use the fact that $P_\eps=p_\eps \otimes p_\eps$ where $p_\eps$ is the heat semigroup on $\Omega$ along with~\cref{ass:potential0} to argue that the first term on the right hand side of the above expression is non-negative.  The desired inequality then follows from ~\cref{ass:potential} (1), that is to say
$$
|W_\eps|(x,x)\le (P_\eps |W|)(x,x)\le C_2|\ln\eps|,
$$
for $C_2$ large enough and $\eps<1/2$.
  
To prove (a), we start by manipulating the expression to obtain
\begin{align}
    \mathring{I}_W[\eta^N_x]&=\mathring{I}_W[\eta^N_x]- \mathring{I}_{W_\eps} [\eta^N_x]+N\mathring{I}_{W_\eps}[\eta^N_x]\\
    &\ge \underbrace{\mathring{I}_W[\eta^N_x]- \mathring{I}_{W_\eps}[\eta^N]}_{=:I}-C_2|\ln\eps|,
\end{align}
where we have used the bound in (b). Expanding the difference, we arrive at the following expression
\begin{align}
    I=&\frac{1}{2N} \sum_{i\ne j} W(x_i,x_j)-W_\eps(x_i,x_j)-\sum_{i=1}^N\int_\Omega \left(W(x_i,y)-W_\eps(x_i,y) \right)\dx \bar\rho(y)\\
    &+\frac{N}{2}\int_{\Omega^2} \left(W(x,y)-W_\eps(x,y)\right)\dx\bar\rho(x)\dx\bar\rho(y) \, .
\end{align}

We control the first term on the right hand side using the super-harmonicity bound~\eqref{superharmonicity} and the other two terms using the Besov estimate \eqref{eq:Besovestimate} to obtain
$$
I\ge -C(N\eps^\alpha+N\eps^\kappa) \, ,
$$
for some fixed constant $C<\infty$. Using the above bound, we obtain
$$
\mathring{I}_W[\eta^N_x]\ge -C(N\eps^\alpha+N\eps^d)-C_2|\ln\eps|\ge -C_1\ln N,
$$
where the last bound follows from picking
$$
\eps=\min\left\{\frac{1}{N^{\frac{1}{\min\{\alpha,\kappa\}}}},\frac12 \right\}
$$
and choosing $C_1$ to be sufficiently large.
\end{proof}

We now move on to the statement and proof of the correlation inequality which forms the technical core of the argument.
\begin{lemma}[Correlation inequality]\label{lem:corineq}
Fix $p \in \N$ and $\gamma \in [1/2,1)$, such that $1\le p\le N$. Let $G:\Omega^2 \to \R$ be such that $G(x,y)=G(y,x)$ for all $x,y \in \Omega$,
\begin{equation}
  \sup_{x\in \Omega}\lVert G(x,\cdot) \rVert_{L^p(\Omega;\bar\rho)} <\infty\,, \quad \textrm{ and }\quad \int_{\Omega}G(x,y) \, \dx{\bar\rho}(y)=0\,,\, \forall x \in \Omega\, .
\end{equation}

Then, there exists a constant $C_p<\infty$, such that
\begin{align}
&\mathbb{{E}}\left[\left|\frac{1}{N}\sum_{i\neq j}^{N}G(X_{i},X_{j})\right|^{p}\right]\\&\qquad \leq C_p\left(\frac{1}{N^{p-1-\lfloor \gamma p \rfloor}} \sup_{x \in \Omega} \lVert  G(x,\cdot)\rVert_{L^p(\Omega;\bar\rho)}^{p}\, +  \sup_{x \in \Omega} \lVert  G(x,\cdot)\rVert_{L^{2(p-\lceil \gamma p \rceil)}(\Omega;\bar\rho)}^{p} \right)\, ,
\label{eq:corineq}
\end{align}
where $(X_i)_{i=1,\dots,N}$ are independent $\bar\rho$-distributed random variables.
\end{lemma}
\begin{proof}
  Before we can proceed to the main body of the proof, we need to introduce some notation to keep track of the various terms that show up when we expand the term on the left hand side of~\eqref{eq:corineq}. Consider the set  
\begin{equation}
\label{eq:defI}
A:= \{(i,j)\in \{1,\cdots,N\}^2:i\neq j\}\, .
\end{equation}
We call a map $I_p:A \to \N\cup\{0\}$ a $p$-multiindex if 
\begin{align}
\sum_{(i,j)\in A}I_p((i,j))=p\, .
\end{align}
Given a $p$-multiindex $I_p$, we associate to it the following multiplicity vector $M_{I_p}= (m_{1,I_p},\dots,m_{N,I_p})\in (\N\cup\{0\})^N$, where
\begin{equation}
m_{i,I_p}:= \sum_{j =1}^N \left(I_p((i,j)) + I_p((j,i))\right)\, .
\end{equation}
We will also need to use the following restricted set of $p$-multiindices
\begin{equation}
\label{eq:restrictedset}
\mathcal{E}_{p}:= \{I_p \text{ is a }p \text{-multiindex}:m_{i,I_p}\neq 1,\, \forall i \in \{1,\dots,N\}\}\,.
\end{equation}
Given a $p$-multiindex $I_p$, we count how many active variables the multiindex has
\begin{equation}
\mathrm{act}(I_p):=\left|B_{I_p}\right|, \qquad B_{I_p}:=\{i \in {1, \dots,N}:m_{i,I_p}\neq 0\}\,.
\end{equation}
We further refine $\mathcal{E}_p$ by considering the amount of active variables
\begin{equation}
\mathcal{E}_{p,\ell}:= \{I_p \in \mathcal{E}_{p}: \mathrm{act}(I_p)=\ell\}\,,
\end{equation}
and rewrite it into pairwise disjoint union
\begin{equation}
\mathcal{E}_p= \bigcup_{\ell=1}^p \mathcal{E}_{p,\ell}\, ,
\end{equation}
where we have used that if $I_p \in \mathcal{E}_p$ then we can must have  at most $p$ active variables that $\mathrm{act}(I_p) \leq p$. For $I_p\in\mathcal{E}_{p,\ell}$, for any $i_0\in \{1,\dots,N\}$ we can obtain the bound
\begin{equation}\label{boundonmi}
    m_{i_0,I_{p}}\le 2p-2(\ell-1),
\end{equation}
this follows from the bound
$$
m_{i_0,I_{p}}+2(\ell-1) \le \sum_{i=1}^Nm_{i,I_{p}}=\sum_{(i,j)\in A}I_p((i,j))+I_p((j,i))=2p,
$$
where we use that for any active particle $m_{i,I_{p}}\ge 2$ by definition of $I_p\in\mathcal{E}_p$. Finally, we state the following estimate on the cardinality of each set
\begin{equation}
|\mathcal{E}_{p,\ell}|\le  {N\choose\ell}(\ell^2 -\ell)^p\le C_p N^{\ell}
\label{eq:sizeE}\, .
\end{equation}

With these concepts in mind, we proceed to estimate the quantity in the statement of the lemma. We have
\begin{align}
 \, \mathbb{{E}}\left[\left|\frac{1}{N}\sum_{i\neq j}^{N}G(X_{i},X_{j})\right|^{p}\right]
=&\, \frac{{1}}{N^{p}}\int_{\Omega^N}\sum_{I_{p}}\frac{p!}{\prod_{(i,j)\in A}I_p(i,j)! }\prod_{(i,j)\in A}G(x_{i,},x_{j})^{I_p((i,j))}\,\mathrm{{d}}\bar\rho^{\otimes N} \\
=&\, \frac{{1}}{N^{p}}\int_{\Omega^N}\sum_{I_{p}\in\mathcal{{E}}_{p}}\frac{p!}{\prod_{(i,j)\in A}I_p(i,j)! } \prod_{(i,j)\in A}G(x_{i,},x_{j})^{I_p((i,j))}\,\mathrm{{d}}\bar\rho^{\otimes N} \\
\leq & \, \frac{C_p}{N^{p}}\int_{\Omega^N}\sum_{\ell=2}^p\sum_{I_{p}\in\mathcal{{E}}_{p,\ell}} \prod_{(i,j)\in A}|G(x_{i,},x_{j})|^{I_p((i,j))}\,\mathrm{{d}}\bar\rho^{\otimes N} \, .\label{eq:termtosplit}
\end{align}
Note that the restriction $I_p\in\mathcal{E}_p$ in the second equality stems from the fact that 
\begin{equation}
\int_{\Omega}G(x,y)\:\mathrm{{d}\bar\rho}(x)=\int_{\Omega}G(x,y)\:\mathrm{{d}\bar\rho}(y)=0 \, ,
\end{equation}
from which it follows that the corresponding terms in the sum are zero. Given $I_{p} \in \mathcal{E}_{p,\ell}$, enumerating the elements of $B_{I_p}$ as $\{i_1,\dots,i_\ell\}$, we define the set of active indices that contain $i_k$ by
\begin{equation}
A_{k,I_p}:= \left\{ (i,j) \in A: (i =i_k \;\mbox{or}\; j=i_k)\;\mbox{and}\; I_p(i,j)>0 \right\} \, \qquad\mbox{for $k=1,\dots,\ell$.}
\end{equation} 
For any $I_p \in \mathcal{E}_{p,\ell}$ we can integrate out the $N-\ell$ variables that do not lie in $B_{I_p}$. We then decompose the remaining variables into a disjoint union of the sets $C_{k,I_p}\coloneqq A_{k,I_p}\setminus \bigcup_{m=1}^{k-1}A_{m,I_p}$ for $k=1,\dots,\ell-1$ to derive the identity
\begin{align}
&\, \int_{\Omega^N}\prod_{(i,j)\in A}|G(x_{i,},x_{j})|^{I_p((i,j))}\,\mathrm{{d}}\bar\rho^{\otimes N}\\
= & \, \int_{\Omega^{\ell}}\prod_{k=1 }^{\ell-1}\left(\prod_{(i,j)\in C_{k,I_p}}|G(x_{i,},x_{j})|^{I_p((i,j))}\right)\mathrm{{d}}\bar\rho^{\otimes \ell}(x_{i_1},...,x_{i_\ell})\label{eq:breakingupA}
\end{align}

To estimate \eqref{eq:breakingupA}, we proceed inductively for each element in $k=1,...,\ell-1$. We first introduce some notation and bounds.
\begin{align}
\gamma_k:=& \,\sum_{(i,j)\in C_{k,I_p}}I_p((i,j))\\\leq&\,\sum_{(i,j)\in A_{k,I_p}}I_p((i,j))\\
\eqqcolon&\, m_{k,p}\le \begin{cases}
    p&\ell<\frac{p}{2}+1\\
    2p-2(\ell-1)& \ell\ge \frac{p}{2}+1
\end{cases}
\, ,
\end{align}
where for the last inequality we have applied the bounds from \eqref{boundonmi}.
We start by bounding the integrand that only depends on $x_{i_1}$, and apply H\"older's inequality with exponents $p_j=\gamma_1/I_p(i_1,j)$ to obtain
\begin{align}
&\int_{\Omega}\prod_{(i,j)\in C_{1,I_p}}|G(x_{i,},x_{j})|^{I_p((i,j))}\mathrm{{d}}\bar\rho(x_{i_1})\\
\le & \, \prod_{(i,j)\in C_{1,I_p}} \left(\int_\Omega |G(x_{i_1},x_{j})|^{\gamma_1}\;\dx \bar\rho(x_{i_1})\right)^{I_p(i_1,j)/\alpha_1}\\
\le & \, \prod_{(i,j)\in C_{1,I_p}}\sup_x \left(\int_\Omega |G(x_{i_1},x)|^{\gamma_1}\;\dx \bar\rho(x_{i_1})\right)^{I_p(i_1,j)/\alpha_1}\\
= & \,\sup_x\int_\Omega |G(x_{i_1},x)|^{\gamma_1}\;\dx \bar\rho(x_{i_1})\\
\leq &   \,
 \begin{cases}
  \displaystyle\sup_{x \in \Omega} \lVert  G(x,\cdot)\rVert_{L^p(\Omega;\bar\rho)}^{\gamma_1} & \ell \leq \frac p2 +1 \\
  \displaystyle \sup_{x \in \Omega} \lVert  G(x,\cdot)\rVert_{L^{2p-2(\ell-1)}(\Omega;\bar\rho)}^{\gamma_1} & \ell > \frac p2 +1,
  \end{cases}
\end{align}
which is a bound that is independent of the other variables $x_{i_2},...,x_{i_\ell}$. We proceed in a similar manner for the other terms in the product, i.e. we can apply H\"older's inequality with exponents $\gamma_k/I_p((i,j))$ to obtain
\begin{align}
 &  \int_{\Omega}\left(\prod_{(i,j)\in C_{k,I_p}}|G(x_{i,},x_{j})|^{I_p((i,j))}\right) \, \dx{\bar\rho}(x_{i_k}) \\
\leq & \,  \, \sup_{x \in \Omega} \lVert  G(x,\cdot)\rVert_{L^{\gamma_k}(\Omega;\bar\rho)}^{\alpha_k}\\
 \leq &  \,
 \begin{cases}
  \displaystyle\sup_{x \in \Omega} \lVert  G_{\eps}(x,\cdot)\rVert_{L^p(\Omega;\bar\rho)}^{\gamma_k} & \ell \leq \frac p2 +1 \\
  \displaystyle \sup_{x \in \Omega} \lVert  G_{\eps}(x,\cdot)\rVert_{L^{2p-2(\ell-1)}(\Omega;\bar\rho)}^{\gamma_k} & \ell > \frac p2 +1
 \end{cases}
 \, .
\end{align}
Putting these bounds together we obtain that for any multiindex $I_p\in \mathcal{E}_{p,l}$, we have
\begin{align}
    &\int_{\Omega^N}\prod_{(i,j)\in A}|G(x_{i,},x_{j})|^{I_p((i,j))}\,\mathrm{{d}}\bar\rho^{\otimes N}\\\le & \,  \begin{cases}\displaystyle
  \prod_{k=1}^{\ell-1}\sup_{x \in \Omega} \lVert  G(x,\cdot)\rVert_{L^p(\Omega;\bar\rho)}^{\gamma_k} & \ell \leq \frac p2 +1 \\
  \displaystyle\prod_{k=1}^{\ell-1}\sup_{x \in \Omega} \lVert  G(x,\cdot)\rVert_{L^{2p-2(\ell-1)}(\Omega;\bar\rho)}^{\gamma_k} & \ell > \frac p2 +1
  \end{cases}\\
  =& \,\begin{cases}\displaystyle
  \sup_{x \in \Omega} \lVert  G(x,\cdot)\rVert_{L^p(\Omega;\bar\rho)}^{p} & \ell \leq \frac p2 +1 \\
  \displaystyle\sup_{x \in \Omega} \lVert  G(x,\cdot)\rVert_{L^{2p-2(\ell-1)}(\Omega;\bar\rho)}^{p} & \ell > \frac p2 +1
    \end{cases}\, ,
\end{align}
where we have used the fact that $\sum_{k=1}^{\ell-1}\gamma_k=p$.
Using the above bound, we see that the right hand side of~\eqref{eq:termtosplit} can be estimated as follows
\begin{align}
& \,\frac{{C_p}}{N^{p}}\int_{\Omega^N}\sum_{\ell=2}^p\sum_{I_{p}\in\mathcal{{E}}_{p,\ell}}\prod_{(i,j)\in A}|G(x_{i,},x_{j})|^{I_p((i,j))}\,\mathrm{{d}}\bar\rho^{\otimes N} \\
\leq & \, \frac{C_p}{N^{p}} \sum_{\ell=1}^{\lfloor\frac p2\rfloor +1}|\mathcal{E}_{p,\ell}| \sup_{x \in \Omega} \lVert  G(x,\cdot)\rVert_{L^p(\Omega;\bar\rho)}^{p}\\
& \, + \frac{C_p}{N^{p}} \sum_{\ell=\lceil\frac p2\rceil +1}^{p}|\mathcal{E}_{p,\ell}| \sup_{x \in \Omega} \lVert  G(x,\cdot)\rVert_{L^{2p-2(\ell-1)}(\Omega;\bar\rho)}^{p}\\
\leq & \, \frac{{C_p}}{N^{p}} \sum_{\ell=2}^{\lfloor\frac p2\rfloor +1} N^{\ell}\sup_{x \in \Omega} \lVert  G_{\eps}(x,\cdot)\rVert_{L^p(\Omega;\bar\rho)}^{p}\\
& \, + \frac{{C_p}}{N^{p}} \sum_{\ell=\lceil\frac p2\rceil +1}^{p} N^{\ell} \sup_{x \in \Omega} \lVert  G(x,\cdot)\rVert_{L^{2p-2(\ell-1)}(\Omega;\bar\rho)}^{p}\label{eq:lastineq} \, ,
\end{align}
where we have used~\eqref{eq:sizeE} along with the standard upper bound on the binomial coefficient and we continue to denote the new constant by $C_p$. We now pick some $\gamma \in [\frac12,1)$ and bound the right hand side of~\eqref{eq:lastineq} as follows
\begin{align}
&\,\frac{{C_p}}{N^{p}} \sum_{\ell=2}^{\lfloor\gamma p\rfloor +1}N^{\ell}\sup_{x \in \Omega} \lVert  G(x,\cdot)\rVert_{L^p(\Omega;\bar\rho)}^{p}\, + \frac{{C_p}}{N^{p}} \sum_{\ell=\lceil\gamma p\rceil +1}^{p} N^{\ell} \sup_{x \in \Omega} \lVert  G(x,\cdot)\rVert_{L^{2p-2(\ell-1)}(\Omega;\bar\rho)}^{p} \\
\leq & \,C_p\left( \frac{1}{N^{p-1-\lfloor \gamma p \rfloor}} \sup_{x \in \Omega} \lVert  G(x,\cdot)\rVert_{L^p(\Omega;\bar\rho)}^{p}\, +  \sup_{x \in \Omega} \lVert  G(x,\cdot)\rVert_{L^{2(p-\lceil \gamma p \rceil)}(\Omega;\bar\rho)}^{p}\right) \, ,
\end{align}
where we continue to denote the constant by $C_p$. This completes the proof of the lemma.
\end{proof}

We finally have the ingredients needed to complete the proof of~\cref{thm:main}.
\begin{proof}[Proof of~\cref{thm:main}]
  As mentioned earlier, the proof is similar in spirit to Nelson's construction of $\varphi^4_2$ (see \cite{hairer2021introduction}). We rewrite $Z_{N}$ as follows
  \begin{equation}
  Z_{N} = 
  \int_0^\infty \mathbb{P}(M(X) \leq -\ln t)\, \mathrm{d}t \, , 
  \end{equation}
  where the random variable $M(X)$ is defined as follows
  \begin{equation}
     M(X) \coloneqq \mathring{I}_{W}[\eta^N_X] \, ,
  \end{equation}
  and, as before, $X=(X_i)_{i=1,\dots,N}$ is a vector of $N$ independent $\bar\rho$-distributed random variables. We can now use~\cref{lem:coulombapriori} (a) to obtain
\begin{equation}
Z_{N} = 1+ \int_1^{N^{C_1}} \mathbb{P}[M(X) \leq -\ln t]\, \mathrm{d}t\, .
\end{equation}
We will now estimate the second term on the right hand side of the above expression. Before we do this, we introduce the random variable
\begin{equation}
  M_{\eps}(X)\coloneqq \mathring{I}_{W_\eps}[\eta^N_X]
\end{equation}
for some $\eps>0$ to be appropriately chosen. We then have
\begin{align}
\int_1^{N^{C_1}} \mathbb{P}[M(X) \leq -\ln t]\, \mathrm{d}t 
=& \,\int_{1}^{N^{C_{1}}}\mathbb{{P}}[M(X)-M_{\eps}(X)\leq-M_{\eps}(X)-\ln t]\;\mathrm{{d}}t\\
\leq & \, \int_{1}^{N^{C_{1}}}\mathbb{{P}}[M(X)-M_{\eps}(X)\leq-1]\;\mathrm{{d}}t\\
\leq& \,\int_{1}^{N^{C_{1}}}\mathbb{{E}}[|M_{\eps}(X)-M(X)|^{p}]\:\mathrm{{d}}t\,, \label{eq:timeintegral}
\end{align}
where we have chosen $\eps=(\frac{e}{t})^{\frac{1}{C_2}}$ and applied ~\cref{lem:coulombapriori} (b) for the first inequality, and Chebyshev's inequality for some $p \in \N$ which have yet to choose for the second inequality. Proceeding, we compute
\begin{align}
\mathbb{{E}}(|M(X)-M_\eps(X)|^{p})= & \, \mathbb{{E}}\left[\left|\int_{\Omega^{2}\setminus\Delta}(W(x,y)-W_{\eps}(x,y))\:\mathrm{{d}}(\eta_{X}^{N})^{\otimes2}(x,y)\right|^{p}\right]\\
=& \, \mathbb{{E}}\left[\left|\frac{1}{N}\sum_{i\neq j}^{N}G_{\eps}(X_{i},X_{j})\right|^{p}\right]\:, \label{eq:intermezzo}
\end{align}
where
\begin{equation}
G_{\eps}(x,y)=\int_{\Omega^2}(W(z_{1},z_{2})-W_{\eps}(z_{1},z_{2}))\mathrm{\mathrm{{d}}(\delta_{x}-\bar\rho)\otimes(\delta_{y}-\bar{\rho})(z_{1},z_{2})}\: .
\end{equation}
It is straightforward to check that $G$ satisfies the conditions in the statement of~\cref{lem:corineq}. Applying~\cref{lem:corineq} we obtain for any $1\leq p<N,\; p \in \N$ and $\gamma \in [1/2,1)$, the following estimate
\begin{eqnarray}
  &&\mathbb{{E}}\left[\left|\frac{1}{N}\sum_{i\neq j}^{N}G_{\eps}(X_{i},X_{j})\right|^{p}\right]\nonumber\\
  &&\qquad\leq C_p\left( \frac{1}{N^{p-1-\lfloor \gamma p \rfloor}} \sup_{x \in \Omega} \lVert  G_{\eps}(x,\cdot)\rVert_{L^p(\Omega;\bar{\rho})}^{p}\, +  \sup_{x \in \Omega} \lVert  G_\eps(x,\cdot)\rVert_{L^{2(p-\lceil \gamma p \rceil)}(\Omega;\bar{\rho})}^{p}\right)\, .\label{eq:finalcor}
\end{eqnarray}
Using \cref{ass:potential}, the Besov estimate \eqref{eq:Besovestimate}, we have
\begin{align}
&\,\sup_{x \in \Omega} \lVert  G_{\eps}(x,\cdot)\rVert_{L^p(\Omega;\mu)}^{p} \leq  \,4 \sup_{x \in \Omega}\int_{\Omega}\left|W(x,y)-W_\eps(x,y)\right|^p \dx{\bar{\rho}}(y)
\leq  \, 4C_{p,\kappa} \eps^{\kappa} \, .
\end{align}
 Applying this estimate to~\eqref{eq:finalcor}, leaves us with
 \begin{equation}
   \mathbb{{E}}\left[\left|\frac{1}{N}\sum_{i\neq j}^{N}G_{\eps}(X_{i},X_{j})\right|^{p}\right]\leq C\left(\frac{\eps^\kappa}{N^{p-1-\lfloor \gamma p \rfloor}}+ \eps^{\kappa\left(\frac{p}{2(p-\lceil \gamma p \rceil)}\right)}\right)\, ,\label{eq:last}
 \end{equation}
 forr some constant $C$ which only depends on $p,\kappa$. We now choose
 \begin{equation}
\bar{\gamma}=\min \left\{\frac{C_2+1-\kappa}{C_2+1},\frac12\right\}\,, \quad \bar{p}= \frac{C_1}{1-\bar{\gamma}}
 \end{equation}
 which ensures that
 \begin{equation}
   \bar{p}-1-\lfloor \bar{\gamma} \bar{p} \rfloor> C_1 \, , \quad \kappa\left(\frac{\bar{p}}{2(\bar{p}-\lceil \bar{\gamma} \bar{p} \rceil)}\right)\geq C_2+1\, .
 \end{equation}
 Furthermore, these choices of $\bar{p},\bar{\gamma}$ are clearly independent of $N$. For $N>\bar{p}$, we can plug in the right hand side of~\eqref{eq:last} into right hand side of~\eqref{eq:intermezzo} and \eqref{eq:timeintegral} to get the integral is bounded. The proof that $Z_{N}$ is bounded for $N<\bar{p}$ is trivial since this is a finite set. Combining the two cases completes the proof of the theorem.
\end{proof}

\section{Application to the Gibbs measure} 
\label{sub:gibbs_measure_and_convergence}
Our main result \cref{thm:main} has an immediate implication for the convergence of the dynamics \eqref{particles} at equilibrium. Under enough regularity and growth conditions on the potentials $V$ and $W$, the system of SDEs~\eqref{particles} admits a unique invariant measure, which is the so-called Gibbs measure $M_N$ defined as follows
\begin{equation}
  \label{eq:Gibbs}
  M_N(\dx{x}) \coloneqq \tilde{Z}_N^{-1} \exp \left(-H_N(x)\right) \mathrm{d}x\,, \quad \tilde{Z}_N\coloneqq \int_{\Omega^N} \exp \left(-H_N(x)\right)\dx{x} \, ,
\end{equation}
where the Hamiltonian is given by
$$
H_N(x)=\frac{1}{2N}\sum_{i\ne j}W(x_i,x_j) + \sum_{i=1}^N V(x_i) \, .
$$
The measure $M_N$ is also the unique minimiser of 
$$
E_N[\rho^N]=\mathcal{H}[\rho^N|\dx{x}] + \int_{\Omega^N}  H_N(x)\, \dx{\rho^N}(x) \, .
$$
We are interested in studying the behaviour of $M_N$ as $N \to \infty$. It is known that in the appropriate sense 
$$
N^{-1}E_N\to^\Gamma E[\rho]:=\mathcal{H}[\rho|\dx{x}] + \int_{\Omega} V \, \dx{\rho} + \frac{1}{2}\int_{\Omega^2} W(x,y)\;\dx\rho(x)\dx\rho(y) \, .
$$ 
We refer the reader to \cite{messer1982statistical} for a classical reference, to \cite{CARRILLO2020108734} for the convergence of the gradient flow structures, and the authors work \cite{Delgadino2021,Delgadino2023} for applications of this convergence to the dynamics. 

In the setting in which $E$ has a unique minimiser, say $\bar \mu \in \cP_2(\Omega)$, this readily implies that 
\begin{equation}
  \overline{\mathcal{H}}[\bar\mu^{\otimes N}|M_N] =o(1) \, . 
\end{equation}
It is also known that one can further quantify the result above to the sharp rate 
\begin{equation}
  \overline{\mathcal{H}}[\bar\mu^{\otimes N}|M_N] =O(N^{-1}) \, ,
  \label{eq:sharprate}
\end{equation}
under smoothness assumptions on the potentials $V$ and $W$, and non-degeneracy of the minimiser $\bar \mu$ (see for instance Ben Arous and Brunaud \cite{BenArous1990} or Ben Arous and Zeitouni \cite{BenArous1999}). In this article, we show that the above result can be extended to repulsive interaction potentials with a logarithmic singularity, such as the repulsive log gas case in arbitrary dimension.

For the sake of brevity, we provide a formal sketch of the argument omitting details. A detailed proof in the repulsive-attractive setting will be presented in the follow-up work~\cite{Delgadino2025}.  We will assume that the potentials $V$ and $W$ are such that the Gibbs measure $M_N$ is well-defined, and $E$ admits a unique minimizer $\bar{\mu}$. We note that since $\bar{\mu}$ minimises $E$ it satisfies the corresponding Euler--Lagrange equation  which can be expressed as
\begin{equation}
  \frac{\dx{\bar{\mu}}}{\dx{x}} =Z_{\bar{\mu}}^{-1} e^{-W\star \bar \mu -V}\, , \quad Z_{\bar{\mu}}=\int_{\Omega} e^{-W\star \bar \mu -V}\, \dx{x}\, .
\end{equation}
Using the above expression it is possible to rewrite the Gibbs measure $M_N$ in the following manner
\begin{equation}
  M_N (\dx{x})=\frac{e^{-\mathring{I}_W[\eta^N_x]}}{Z_{N}} \dx \bar\mu^{\otimes N}(x)
\end{equation}
where now
\begin{align}
 \qquad Z_{N}= &\, e^{N E[\bar{\mu}]}\tilde{Z}_N = \, \int_{\Omega^N}e^{-\mathring{I}_W[\eta^N_x]} \;\dx \bar\mu^{\otimes N}(x) \, .
\end{align}
Using this new form of the Gibbs measure, we can compute the relative entropy of the Gibbs measure $M_N$ with respect to the product measure $\bar{\mu}^{\otimes N}$ as follows
\begin{equation}
   \mathcal{H}[\bar{\mu}^{\otimes N}|M_N] = \log Z_{N}\, ,
\end{equation} 
where we have used the fact that
\begin{equation}\label{eq:cancellation}
  \int_{\Omega^N} \mathring{I}_W[\eta^N_x]\, \dx{\bar \mu}^{\otimes N}=0 \, .
\end{equation}
Similarly, the relative entropy of $M_N$ with respect to the product measure $\bar{\mu}^{\otimes N}$ can be expressed as
\begin{eqnarray}
  \mathcal{H}[M_N|\bar{\mu}^{\otimes N}] &=&  -\frac{1}{Z_{N}}\int_{\Omega^N}\mathring{I}_W[\eta^N_x] \, e^{-\mathring{I}_W[\eta^N_x]}\dx{\bar{\mu}^{\otimes N}}(x)  -\log Z_{N}\nonumber\\
   &\le& \frac{\int_{\Omega^N}|\mathring{I}_W[\eta^N_x]|^2\;\dx{\bar{\mu}^{\otimes N}}(x)}{Z_{N}}+ \frac{\int_{\Omega^N} e^{-2\mathring{I}_W[\eta^N_x]}\dx{\bar{\mu}^{\otimes N}}(x)}{Z_N}  -\log Z_{N}\,\nonumber \\
    &=& \frac{\int_{\Omega^2} W^2(x,y)\;\dx\bar\mu(x)\dx\bar\mu(y)}{Z_N}+\frac{\int_{\Omega^N} e^{-2\mathring{I}_W[\eta^N_x]}\dx{\bar{\mu}^{\otimes N}}(x)}{Z_N}  -\log Z_{N}\, ,
\label{eq:entropybound}
\end{eqnarray}
where we have used Cauchy--Schwarz inequality in the first line and \eqref{eq:cancellation} in the second line. Controlling the first term in \eqref{eq:entropybound} is straightforward, the second term is the partition function $Z_{N,2}$, and the last term is the logarithm of the partition function $Z_N$. We are ready to state a corrollary of the bound of the partition function in \cref{thm:main}.

\begin{theorem}[Convergence of the Gibbs measure]\label{thm:Gibbsconvergence}
  Assume $W$ is such that ~\cref{ass:potential0,ass:potential} are satisfied and $V,\,W$ are chosen such that the mean field energy $E$ has a unique minimiser $\bar{\mu} \in \cP_2(\Omega)$ with an essentially bounded density. Then, we have that
    \begin{equation}
    \sup_{N \in \N} Z_{N} < \infty \, .
  \end{equation}
    In particular, we have that
    \begin{equation}
    \overline{\mathcal{H}}[M_N|\bar{\mu}^{\otimes N}] \le \frac{C}{N}, \qquad \overline{\mathcal{H}}[\bar{\mu}^{\otimes N}|M_N]\le \frac{C}{N} \, .
  \end{equation}
  Moreover, if $\bar{\mu}$ satisfies a Talagrand inequality, 
  \begin{equation}
    \frac{1}{N}d_2^2(M_N,\bar{\mu}^{\otimes N})\le \frac{C}{N} \, .
  \end{equation}
\end{theorem}




\appendix
\appendix
\section{Estimates for the repulsive log gas} 
\label{sec:estimates_for_the_repulsive_log_gas}
In this section, we will check that the repulsive log gas satisfies~\cref{ass:potential0,ass:potential}. We have the following result:
\begin{lemma}
  Let $d\geq 1$ and define $F:\T^d \to \R$
  \begin{equation}
    F(x)\coloneqq(-\Delta)^{-\frac d2} (\delta_0 -1) \coloneqq \sum_{k \in \Z^d\setminus \{0\}}\frac{1}{|2 \pi k|^d}e^{i 2 \pi k\cdot x} \, , 
  \end{equation}
  and set $W(x,y)=F(x-y)$. Then, $W$ satisfies~\cref{ass:potential0,ass:potential}. 

  Similarly, for  $d\geq 1$, define $W:\R^d \times \R^d \to \R$
  \begin{equation}
    W(x,y)=-\ln (|x-y|) \, .
  \end{equation}
  Then, again $W$ satisfies ~\cref{ass:potential0,ass:potential}.
\end{lemma}
\begin{proof}
 We will use $p_\eps$ to denote the heat semigroup on $\R^d$ or $\T^d$. The proof of ~\cref{ass:potential0} for $F$ is straightforward since the Fourier coefficients of $F$ are all non-negative. $F$ also satisfies~\cref{ass:potential} (a), since
  \begin{equation}
    \lVert P_\eps F \rVert_{L^\infty(\T^d)} \leq \sum_{k \in \Z^d \setminus \{0\}}\frac{e^{-|2\pi k|^2 \eps^2}}{|2 \pi k|^d} \leq C_0 (1 + |\ln (|\eps|)|)\, .
  \end{equation}
  We now move on the~\cref{ass:potential} (b). For any $\mu$ with bounded density and $p>d$, we have the bound
  \begin{align}
    \lVert p_\eps W(x,\cdot) - W(x,\cdot) \rVert_{L^p(\T^d;\bar{\rho})}^p \leq & \, \left\lVert \frac{\dx{\bar{\rho}}}{\dx{x}}\right \rVert_{L^\infty(\T^d)} \lVert p_\eps F - F \rVert_{L^p(\T^d)}^p \\
    =& \, C \int_{\T^d}\left | \int_{\T^d} p_\eps(y)(F(y-x)-F(x)) \, \dx{y}\right|^p \, \dx{x}\\
    \leq & \,\int_{\T^d}p_\eps(y) \int_{\T^d} |F(y-x)-F(x)|^p \, \dx{y} \, \dx{x}\\
    \leq & \, C \int_{\T^d}|y|^{ d}p_\eps(y)\dx{y} \lVert F\rVert_{\dot{B}_{p,\infty}^{\frac{d}{p}}}^p  \leq C \eps^{\frac{d}{2}}\lVert F\rVert_{\dot{B}_{p,\infty}^{\frac{d}{p}}}^p\, ,
  \end{align}
  where in the penultimate inequality we have applied~\cite[Theorem 2.36]{BCD11}. The homogeneous Besov norm  which shows up on the right hand side can then be expressed as follows
  \begin{equation}
    \lVert F\rVert_{\dot{B}_{p,\infty}^{\frac{d}{p}}} = \sup_{j \in \Z} 2^{jd/p} \lVert \dot{\Delta}_j F \rVert_{L^p(\T^d)}\, .
  \end{equation}
   Note now that by~\cite[Lemma 2.1]{BCD11}
  \begin{align}
    \lVert \dot{\Delta}_j F \rVert_{L^p(\T^d)} \leq \, C 2^{j d(\frac{1}{2}-\frac1p)} \lVert \dot{\Delta}_j F \rVert_{L^2(\T^d)} \leq  \, C2^{-jd/p} \, ,
  \end{align}
  from which the result follows. We are left to check that~\cref{ass:potential} (c) is satisfied. For this, we note that
  \begin{align}
       &W(x,y)-(P_\eps)W(x,y)=\, -\int_0^\eps \frac{\dx}{\dx{s}}(p_sF)(x-y)\, \dx{s} =\,  -\int_0^\eps (p_s\Delta F)(x-y)\, \dx{s}\, .
  \end{align}
  If $d=2$, the above quantity is clearly non-negative. A similar argument works in $d=1$ with $P_\eps,\,p_\eps$ replaced by $P_\eps^{\frac12},\,p_\eps^{\frac12}$. For $d>2$, we have that 
  \begin{align}
      -\Delta F = & \, (-\Delta)^{\frac d2-1} (\delta_0 -1)\\
      = &\,\frac{1}{\Gamma(\frac d2 -1)}\int_0^\infty s^{\frac d2 -2} (p_s(x)-1) \, \dx{s}\\
      = & \, \frac{1}{\Gamma(\frac d2 -1)}\int_0^1 s^{\frac d2 -2} (p_s(x)-1) \, \dx{s} + \frac{1}{\Gamma(\frac d2 -1)}\int_1^\infty s^{\frac d2 -2} (p_s(x)-1) \, \dx{s} \, \\
      \geq &\, \frac{1}{\Gamma(\frac d2 -1)} - C_d\frac{1}{\Gamma(\frac d2 -1)}\int_1^\infty s^{\frac d2 -2} e^{-\lambda_d s} \, \dx{s}\geq -C\, ,
  \end{align}
  where we have used both that $p_s$ is non-negative and that it converges exponentially to $1$ uniformly in $x\in \T^d$. This completes the proof of~\cref{ass:potential0} (3).

  We now move on the second part of the lemma, i.e. the choice $W(x,y)=-\ln(|x-y|)$. We start as before with~\cref{ass:potential0}.  Fix $\eta$ to be a finite signed measure with $\eta(\R^d)=0$ and let $(\eta_n)_{n \in \N}$ be a family of Schwartz functions which converge weakly to $\eta$ and have zero mean $\eta_n(\R^d)=0$ such that $I_{P_\eps W}[\eta_n]$ converges to $I_{P_\eps W}[\eta]$. Then, using the fact that,
  \begin{align}
    I_{P_\eps W}[\eta_n]=C_d\int_{\R^d}\frac{e^{-|2\pi k |^2 \eps^2}}{|k|^{d}}|\hat{\eta}_n(k)|^2 \, \dx{k} \geq 0\, ,
  \end{align}
  for some $C_d\in (0,\infty)$, we have that $I_{P_\eps W}[\eta]\geq 0$.

  We now move on to~\cref{ass:potential} (1). We have
  \begin{equation}
    |P_\eps W|(x,x) \leq \int_{\T^d}p_\eps (x)|\ln (|x|)|\dx{x}\leq C_0(1+ |\ln|\eps||)\, .
  \end{equation}
  For~\cref{ass:potential} (2), we similarly have
  \begin{align}
    \lVert P_\eps W(x,\cdot) - W(x,\cdot) \rVert_{L^p(\R^d;\bar{\rho})}^p &\leq  \left\lVert \frac{\dx{\bar{\rho}}}{\dx{x}}\right \rVert_{L^\infty(\T^d)} \int_{\R^d}\left|\int_{\R^d}p_\eps(x-y)\ln(|y|/|x|)\, \dx{y}\right|^p\, \dx{x}\\&\leq \, C\eps^\kappa \, ,
  \end{align}
  for some $\kappa>0$. Finally, we treat~\cref{ass:potential} (3). Note that
  \begin{eqnarray*}
      \, W(x,y)-(P_\eps)W(x,y)&=&\,\displaystyle -\int_0^\eps \frac{\dx}{\dx{s}}(P_s W)(x,y)\, \dx{s} \\
    &=&\displaystyle\,\int_0^\eps \frac{\dx}{\dx{s}}(p_s (\ln(|\cdot|)))(x-y)\, \dx{s} \\
    &=&\displaystyle\, \int_0^\eps (\Delta p_s (\ln(|\cdot|)))(x-y)\, \dx{s}\\
    &=&\displaystyle\,
    \begin{cases}
      2\pi \int_0^\eps p_s (x-y)\, \dx{s}& d=2 \\
      \int_0^\eps p_s \left(\frac{d-2}{|\cdot|^{2}}\right)(x-y)& d\geq 3
    \end{cases} \\
    &\geq & \, 0 \,.
  \end{eqnarray*}
 Again, a similar argument works for $d=1$  with $P_\eps,p_\eps$ replaced by $P_\eps^{\frac12},p_\eps^{\frac12}$.

\end{proof}

\section*{Acknowledgements}
The authors would like to thank P-E Jabin and M Rosenzweig for stimulating discussions and pointing out the related reference of Grotto and Romito. The authors would also like to thank S. A. Smith for directing us to Nelson's original argument and G. Pavliotis for constant encouragement.  The research of MGD was partially supported by NSF-DMS-2205937. The research of RSG was partially supported by the Deutsche Forschungsgemeinschaft through the SPP 2410/1 \emph{Hyperbolic Balance Laws in Fluid Mechanics: Complexity, Scales, Randomness}.

\bibliographystyle{amsplainabbrv}
\bibliography{ref}
\end{document}